\documentclass[11pt]{article}

\usepackage[a4paper,margin=1in]{geometry}

\usepackage[T1]{fontenc}
\usepackage{libertinus}

\usepackage{microtype}

\usepackage{amsmath,amssymb,amsthm,mathtools}
\usepackage{mathrsfs}

\usepackage{enumitem}
\usepackage{xcolor}
\usepackage{comment}
\usepackage{url}
\usepackage[hidelinks]{hyperref}

\newcommand{\R}{\mathbb{R}}
\newcommand{\Id}{\mathrm{Id}}
\newcommand{\E}{\mathbb{E}}

\DeclareMathOperator{\Cov}{Cov}
\DeclareMathOperator{\Var}{Var}
\DeclareMathOperator{\Tr}{Tr}

\DeclareMathOperator{\codim}{codim}

\theoremstyle{plain}
\newtheorem{theorem}{Theorem}[section]
\newtheorem{proposition}[theorem]{Proposition}
\newtheorem{lemma}[theorem]{Lemma}
\newtheorem{corollary}[theorem]{Corollary}

\theoremstyle{remark}

\title{A Deterministic Proof of the Slicing Theorem}
\author{Silouanos Brazitikos}
\date{}

\begin{document}

\maketitle
\begin{abstract}
We give a deterministic proof of Bourgain's slicing theorem.
The proof uses neither stochastic localization nor stochastic calculus,
does not rely on Guan's covariance bound, and avoids $M$-ellipsoids,
Pisier's regularization machinery, and complex interpolation.
Instead, the argument is based on a deterministic one-parameter
quadratic perturbation of the measure, a covariance-survival principle,
and a geometric analysis of $L_p$-centroid bodies.
\end{abstract}
\section{Introduction}

A convex body $K\subseteq\R^n$ is said to be in isotropic position if
\[
    |K|=1,
    \qquad
    \int_K x\,dx=0,
\]
and there exists a constant $L_K>0$ such that
\[
    \int_K \langle x,\theta\rangle^2\,dx
    =
    L_K^2
\]
for every $\theta\in S^{n-1}$. The number $L_K$ is called the
isotropic constant of $K$. More generally, for an arbitrary centered
convex body one may define the affine-invariant quantity
\[
    L_K
    =
    |K|^{-1/n}
    \det\Cov(\lambda_K)^{1/(2n)},
\]
where $\lambda_K$ denotes the uniform probability measure on $K$.
We write
\[
    L_n
    =
    \sup\{L_K:K\subseteq\R^n\text{ is a convex body}\}.
\]

Bourgain's slicing problem, also known as the hyperplane conjecture,
asks whether
\[
    L_n\leq C
\]
for a universal constant $C>0$. Equivalently, every convex body
$K\subseteq\R^n$ of volume one should admit a hyperplane section of
$(n-1)$-dimensional volume bounded from below by a universal positive
constant. The problem has played a central role in asymptotic convex
geometry since Bourgain's work in the 1980s
\cite{Bourgain1986,Bourgain1991}. We refer to
\cite{BGVV2014} for a systematic account of the theory of isotropic
convex bodies, and to the recent surveys
\cite{KlartagLehecSurvey2025,GiannopoulosPafisTziotziou2025}
for the developments leading to the resolution of the problem and for
its connections with concentration, isoperimetry, the thin-shell
problem, and the Kannan--Lov\'asz--Simonovits conjecture.

For many years the best general estimates for $L_n$ were polynomial
in the dimension. A sequence of major advances based on stochastic
localization led eventually to polylogarithmic bounds. In particular,
Klartag proved
\[
    L_n\leq C\sqrt{\log n}
\]
in \cite{Klartag2023}. A key ingredient in his argument is the
improved log-concave Lichnerowicz inequality: if $\nu$ is
$t$-uniformly log-concave, then
\[
    C_P(\nu)
    \leq
    \sqrt{
        \frac{\|\Cov(\nu)\|_{\mathrm{op}}}{t}
    }.
\]
Shortly afterwards, Guan in \cite{Guan2024} obtained the striking estimate
\[
    L_n\leq C\log\log n.
\]
Building on Guan's covariance estimate, Klartag and
Lehec completed the affirmative resolution of Bourgain's slicing
problem \cite{KlartagLehec2025}. Their proof combines Milman's theory
of $M$-ellipsoids, Eldan's stochastic localization, Guan's estimate,
and stability estimates for the Shannon--Stam inequality.

Bizeul subsequently gave a different proof of the slicing conjecture
based on exponential small-ball estimates for isotropic log-concave
measures \cite{Bizeul2025}. His argument also uses stochastic
localization together with Guan's bound, but the final passage from
small-ball probabilities to the slicing conjecture follows the
reduction of Dafnis and Paouris through Milman's theory of
$M$-ellipsoids. This provides a conceptually different interpretation
of the information supplied by the stochastic localization process:
a sufficiently strong lower-tail estimate for the Euclidean norm
already forces the isotropic constant to be universally bounded.

The purpose of the present note is to give a genuinely deterministic route to the slicing theorem. The argument does not use stochastic localization, martingales, Itô calculus, or any other form of stochastic calculus, and it is independent of Guan's covariance estimate. It also avoids the theory of M-ellipsoids and regular positions altogether; in particular, no use is made of Pisier's machinery, complex interpolation, or the entropy and covering arguments underlying that theory. The only deformation considered in the proof is a deterministic
one-parameter quadratic perturbation of the measure
\[
    d\mu_t(x)
    =
    \frac{e^{-t|x|^2/2}}{Z_t}\,d\mu(x),
    \qquad
    Z_t
    =
    \int_{\R^n}e^{-t|x|^2/2}\,d\mu(x), \qquad A_t=\Cov(\mu_t)
\]
and the proof proceeds from a direct analysis of the covariance along this one-parameter family, followed by a geometric argument based on $L_p$-centroid bodies.  The idea of
using deterministic perturbations in connection with the slicing
problem has an important precedent in Klartag's work
\cite{Klartag2006ConvexPerturbations}, where convex perturbations were
used in the proof of the bound $L_n\lesssim n^{1/4}$, although the perturbation
used here is of a different nature. 

We prove that there are universal
constants $c_1,c_2,\beta,t_0>0$ and a subspace $F\subseteq\R^n$ with
\[
    \dim F\geq\beta n
\]
such that
\[
    c_1\Id_F
    \preceq
    A_{t_0}P_F
    \preceq
    c_2\Id_F.
\]
The quantity used to detect this phenomenon is 
\[
    \Phi(t)
    =
    \Tr\bigl(\sqrt{\Id+A_t}-\Id\bigr)
    =
    \sum_{i=1}^n
    \left(
        \sqrt{1+\lambda_i(A_t)}-1
    \right).
\]
Since $A_0=\Id$,
\[
    \Phi(0)=(\sqrt2-1)n.
\]
The square-root potential is adapted to the available Poincar\'e
estimate. Differentiation introduces the matrix
\[
    (\Id+A_t)^{-1/2}.
\]
On a spectral block where $A_t$ has size $s$, this contributes a
factor of order $s^{-1/2}$, while the improved Lichnerowicz
inequality contributes the complementary factor
$\sqrt{s/t}$. After a dyadic decomposition of the covariance
spectrum, this matching yields
\[
    |\Phi'(t)|
    \leq
    Cn\,t^{-3/4}.
\]
The singularity at the origin is integrable, and hence, for a
sufficiently small universal $t_0$,
\[
    \Phi(t_0)\geq cn.
\]
Together with the Brascamp--Lieb bound
$A_{t_0}\preceq t_0^{-1}\Id$, this gives the required
result for the subspaces.

From this point the argument has two geometric steps. First, the above implies a dimension-proportional exponential small-ball
estimate for every isotropic log-concave measure. Such a small-ball
estimate is already sufficient to conclude the slicing conjecture via
the reduction of Dafnis and Paouris \cite{DafnisPaouris2010}. We instead give a direct
completion in order to keep the argument essentially self-contained.
Using the small-ball estimate, we obtain lower Hessian bounds for
$L_p$-centroid bodies on subspaces of codimension $O(p)$. A weighted
dyadic combination of these bodies produces a norm whose squared
support function has Hessian eigenvalues satisfying
\[
    \lambda_i\gtrsim i^2.
\]
The resulting determinant estimate is then converted, through
Legendre duality into a
volume estimate. The reverse Santal\'o and Rogers--Shephard
inequalities complete the proof. Note that for the reverse Santal\'o there is a purely
convex geometric proof in \cite{GiannopoulosPaourisVritsiou2014}, which is based on Klartag's work in \cite{Klartag2006ConvexPerturbations}. 

We finally mention a recent development of a different nature. Chen
and Klartag in \cite{ChenKlartag2026} have obtained the sharp thin-shell inequality by a
deterministic argument based on log-concave moment measures, a
weighted Riemannian structure, and the Monge--Amp\`ere equation. 

\section{Preliminaries}
\label{sec:preliminaries}

We collect here the few analytic and geometric facts that will be used
throughout the proof.

\paragraph{Uniform log-concavity and Poincar\'e inequalities.}
A probability measure $\nu$ on $\R^d$ is called
$t$-uniformly log-concave if it has a density of the form
\[
    d\nu(x)
    =
    e^{-t|x|^2/2}g(x)\,dx,
\]
up to normalization, where $g$ is log-concave. Equivalently, if
$d\nu=e^{-V}dx$, then
\[
    x\longmapsto V(x)-\frac{t}{2}|x|^2
\]
is convex. In the smooth case this is equivalent to
\[
    \nabla^2V\succeq t\Id.
\]
We denote by $C_P(\nu)$ the optimal constant in the Poincar\'e
inequality
\[
    \Var_\nu(f)
    \leq
    C_P(\nu)\int |\nabla f|^2\,d\nu.
\]
We shall repeatedly use the Brascamp--Lieb variance inequality
\cite{BrascampLieb1976}. If
\[
    d\nu(x)=Z^{-1}e^{-V(x)}\,dx
\]
with $V$ smooth and strictly convex, then
\begin{equation}
\label{eq:BL-prelim}
    \Var_\nu(f)
    \leq
    \int
    \left\langle
        (\nabla^2V)^{-1}\nabla f,\nabla f
    \right\rangle
    d\nu.
\end{equation}
The general log-concave case follows by approximation. In particular,
if $\nu$ is $t$-uniformly log-concave, then
\begin{equation}
\label{eq:classical-Lichnerowicz}
    C_P(\nu)\leq\frac1t.
\end{equation}
Applying \eqref{eq:BL-prelim} to linear functions also gives
\begin{equation}
\label{eq:BL-cov-prelim}
    \Cov(\nu)\preceq\frac1t\Id.
\end{equation}

The covariance-sensitive improvement of
\eqref{eq:classical-Lichnerowicz} that we need is due to Klartag
\cite[Theorem~1.3]{Klartag2023}:
\begin{equation}
\label{eq:improved-Lichnerowicz-prelim}
    C_P(\nu)
    \leq
    \sqrt{
        \frac{\|\Cov(\nu)\|_{\mathrm{op}}}{t}
    }.
\end{equation}
Notice that \eqref{eq:BL-cov-prelim} recovers from this the classical
bound $C_P(\nu)\leq t^{-1}$. In the proof below the genuinely improved
estimate \eqref{eq:improved-Lichnerowicz-prelim} is used only when the
covariance-sensitive factor
$\|\Cov(\nu)\|_{\mathrm{op}}^{1/2}$ is essential; the estimate
$C_P(\nu)\leq t^{-1}$ itself follows directly from
Brascamp--Lieb.

\paragraph{Gaussian concentration.}
The Bakry--\'Emery criterion \cite{BakryEmery1985} implies that a
$t$-uniformly log-concave probability measure satisfies the
logarithmic Sobolev inequality
\[
    \operatorname{Ent}_\nu(h^2)
    \leq
    \frac{2}{t}
    \int |\nabla h|^2\,d\nu.
\]
The standard Herbst argument then yields the following Gaussian
concentration estimate: for every $1$-Lipschitz function
$F:\R^d\to\R$ and every $s>0$,
\begin{equation}
\label{eq:Gaussian-concentration-prelim}
    \nu\{F\leq \E_\nu F-s\}
    \leq
    \exp\left(-\frac{ts^2}{2}\right),
\end{equation}
and similarly for the upper tail. We refer also to
\cite{Ledoux2001} for this standard implication.
\paragraph{A one-dimensional moment comparison.}
We shall need the following elementary consequence of log-concavity.
If $Z$ is a centered log-concave real random variable and
$p\geq q\geq1$, then
\begin{equation}
\label{eq:moment-comparison-prelim}
    \|Z\|_p
    \leq
    \frac{p}{q}\|Z\|_q.
\end{equation}
See \cite[Theorem~1.2]{Murawski2026}.
\paragraph{Differentiation of spectral functions and max-min theorem}
We shall use the standard differentiation rule for spectral functions
of symmetric matrices, see \cite[Theorem~1.1]{Lewis1996}. Let $\varphi$ be continuously differentiable
on an interval containing the spectrum of a symmetric matrix $A$, and
set
\[
    \mathcal F(A)=\Tr\varphi(A).
\]
Then, for every symmetric matrix $H$,
\begin{equation}
\label{eq:spectral-differentiation-prelim}
    D\mathcal F(A)[H]
    =
    \Tr\bigl(\varphi'(A)H\bigr).
\end{equation}
Moreover, we will use  the
Courant--Fischer formula
\[
    \lambda_{d+1}(H)
    =
    \max_{\substack{E\subseteq\R^n\\\dim E=n-d}}
    \ \min_{\substack{v\in E\\|v|=1}}
    \langle Hv,v\rangle
\]
Finally, we use the standard duality facts for a norm and its dual:
if $R$ is an origin-symmetric convex body and
\[
    G(u)=\frac12h_R(u)^2,
\]
then
\[
    G^*(x)=\frac12p_R(x)^2.
\]
At points of twice differentiability,
\[
    D^2G^*(x)
    =
    \bigl[D^2G(\nabla G^*(x))\bigr]^{-1}.
\]
These are standard facts of Legendre duality; see, for example,
\cite{Rockafellar1970}.
\section{The uniformly log-concave deformation}

Let $\mu$ be an even isotropic log-concave probability measure on
$\R^n$. For $t\geq0$, define
\[
    d\mu_t(x)
    =
    \frac{e^{-t|x|^2/2}}{Z_t}\,d\mu(x),
    \qquad
    Z_t
    =
    \int_{\R^n} e^{-t|x|^2/2}\,d\mu(x),
\]
and write
\[
    A_t=\Cov(\mu_t).
\]
Since $\mu_t$ is even, its barycenter is the origin.

The main quantity that we will use is 
\[
    \Phi(t)
    =
    \Tr\bigl(\sqrt{\Id+A_t}-\Id\bigr).
\]

We first prove the following differential estimate.

\begin{lemma}
\label{lem:central-differential}
For every $0<t\leq1$,
\begin{equation}
\label{eq:Phi-derivative-bound}
    |\Phi'(t)|
    \leq
    Cn\,t^{-3/4}.
\end{equation}
\end{lemma}

\begin{proof}
We first observe that
\begin{equation}
\label{eq:second-moment-decreasing}
    \frac{d}{dt}
    \int_{\R^n}|x|^2\,d\mu_t(x)
    =
    -\frac12
    \int_{\R^n}
    \left(
        |x|^2
        -
        \int_{\R^n}|y|^2\,d\mu_t(y)
    \right)^2
    d\mu_t(x)
    \leq0.
\end{equation}
Since $\mu=\mu_0$ is isotropic,
\begin{equation}
\label{eq:trace-bound-integral}
    \int_{\R^n}|x|^2\,d\mu_t(x)
    \leq
    \int_{\R^n}|x|^2\,d\mu_0(x)
    =
    n.
\end{equation}
On the other hand, $\mu_t$ is $t$-uniformly log-concave. Indeed, if
$d\mu=e^{-V(x)}\,dx$ with $V$ convex, then the potential of $\mu_t$
is
\[
    V_t(x)
    =
    V(x)+\frac{t}{2}|x|^2,
\]
and hence
\[
    \nabla^2V_t\succeq t\Id.
\]
Therefore the Brascamp--Lieb inequality, applied to the linear
function $x\mapsto\langle x,u\rangle$, gives
\[
    \langle A_tu,u\rangle
    =
    \Var_{\mu_t}\langle X,u\rangle
    \leq
    \frac{|u|^2}{t}.
\]
Since this holds for every $u\in\R^n$,
\begin{equation}
\label{eq:BL-covariance}
    A_t\preceq t^{-1}\Id.
\end{equation}
We next differentiate $\Phi$.  Consider
\[
    F(A)
    :=
    \Tr\!\left((I+A)^{1/2}-I\right).
\]
If $\lambda(A)=(\lambda_1(A),\ldots,\lambda_n(A))$, then
\[
    F(A)=g(\lambda(A)),
\]
where
\[
    g(x_1,\ldots,x_n)
    :=
    \sum_{i=1}^n\left(\sqrt{1+x_i}-1\right).
\]
The function $g$ is symmetric and continuously differentiable on
$(-1,\infty)^n$.  Since $A_t$ is positive semidefinite, all its
eigenvalues belong to this interval.  Write
\[
    A_t
    =
    U\operatorname{diag}(\lambda_1,\ldots,\lambda_n)U^T.
\]
By \cite[Theorem~1.1]{Lewis1996},
\begin{align*}
    \nabla F(A_t)
    &=
    U\operatorname{diag}
    \left(
        \frac{1}{2\sqrt{1+\lambda_1}},
        \ldots,
        \frac{1}{2\sqrt{1+\lambda_n}}
    \right)U^T  \\
    &=
    \frac12(I+A_t)^{-1/2}.
\end{align*}
Here the gradient is with respect to the Hilbert--Schmidt inner product.
Consequently, by the chain rule,
\begin{align}
    \Phi'(t)
    &=
    \left\langle\nabla F(A_t),A_t'\right\rangle_{\mathrm{HS}}
    \nonumber\\
    &=
    \frac12
    \Tr\!\left((I+A_t)^{-1/2}A_t'\right),
\label{eq:Phi-derivative}
\end{align}
where we used
$\langle C,D\rangle_{\mathrm{HS}}=\Tr(CD)$ for symmetric matrices. 
Set
\[
    B_t=(\Id+A_t)^{-1/2}
\]
and, for $s\geq0$, define
\[
    H_t(s)=\Tr(B_tA_s).
\]
Since $B_t$ is fixed with respect to $s$,
\[
    \frac{d}{ds}H_t(s)=\Tr(B_tA_s'),
\]
and therefore
\[
    \Phi'(t)
    =
    \frac12
    \left.\frac{d}{ds}H_t(s)\right|_{s=t}.
\]
Since $\mu_s$ is centered,
		\[
		A_s
		=
		\Cov(\mu_s)
		=
		\int_{\R^n} x\otimes x\,d\mu_s(x).
		\]
		Therefore,
		\[
		H_t(s)
		=
		\Tr(B_tA_s)
		=
		\int_{\R^n}
		\Tr\!\bigl(B_t(x\otimes x)\bigr)\,d\mu_s(x).
		\]
		For every $x\in\R^n$,
		\[
		\Tr\!\bigl(B_t(x\otimes x)\bigr)
		=
		\langle B_tx,x\rangle.
		\]
		Hence
		\[
		H_t(s)
		=
		\int_{\R^n}
		\langle B_tx,x\rangle\,d\mu_s(x).
		\]
Using the definition of $\mu_s$, we may write
\[
    H_t(s)
    =
    \frac{
        \displaystyle
        \int_{\R^n}
        \langle B_tx,x\rangle
        e^{-s|x|^2/2}\,d\mu(x)
    }{
        \displaystyle
        \int_{\R^n}
        e^{-s|x|^2/2}\,d\mu(x)
    }.
\]
Consequently,
\begin{align}
    \Phi'(t)
    &=
    \frac12 H_t'(t)
    \nonumber\\
    &=
    \frac12
    \left.
    \frac{d}{ds}
    \frac{
        \displaystyle
        \int_{\R^n}
        \langle B_tx,x\rangle
        e^{-s|x|^2/2}\,d\mu(x)
    }{
        \displaystyle
        \int_{\R^n}
        e^{-s|x|^2/2}\,d\mu(x)
    }
    \right|_{s=t}.
\end{align}
For notational convenience below  we set $B:=B_t$. Differentiating the quotient gives
\begin{align}
\label{eq:Phi-integral}
    \Phi'(t)
    =
    -\frac14
    \Bigg[
        &\int_{\R^n}
        \langle Bx,x\rangle |x|^2\,d\mu_t(x)
        \nonumber\\
        &-
        \left(
            \int_{\R^n}\langle Bx,x\rangle\,d\mu_t(x)
        \right)
        \left(
            \int_{\R^n}|x|^2\,d\mu_t(x)
        \right)
    \Bigg].
\end{align}
Equivalently,
\begin{align}
\label{eq:Phi-centered-integral}
    \Phi'(t)
    =
    -\frac14
    \int_{\R^n}
    &
    \left(
        \langle Bx,x\rangle
        -
        \int_{\R^n}\langle By,y\rangle\,d\mu_t(y)
    \right)
    \nonumber\\
    &\times
    \left(
        |x|^2
        -
        \int_{\R^n}|y|^2\,d\mu_t(y)
    \right)
    d\mu_t(x).
\end{align}
Hence, by Cauchy--Schwarz,
\begin{equation}
\label{eq:Phi-CS}
    |\Phi'(t)|
    \leq
    \frac14 UV,
\end{equation}
where
\begin{equation}
\label{eq:def-U}
    U^2
    =
    \int_{\R^n}
    \left(
        \langle Bx,x\rangle
        -
        \int_{\R^n}\langle By,y\rangle\,d\mu_t(y)
    \right)^2
    d\mu_t(x)
\end{equation}
and
\begin{equation}
\label{eq:def-V}
    V^2
    =
    \int_{\R^n}
    \left(
        |x|^2
        -
        \int_{\R^n}|y|^2\,d\mu_t(y)
    \right)^2
    d\mu_t(x).
\end{equation}
We first estimate $U$. Recall that the covariance ellipsoid $Z_2(\mu_t)$ has support
function
\[
    h_{Z_2(\mu_t)}(u)^2
    =
    \langle A_tu,u\rangle.
\]
We group its principal axes according to their lengths. Let $E_0$
be the span of the principal axes of $Z_2(\mu_t)$ of length at most
$1$, and, for $j\geq1$, let $E_j$ be the span of those whose lengths
belong to
\[
    \left(2^{(j-1)/2},2^{j/2}\right].
\]
Let $P_j$ denote the orthogonal projection onto $E_j$, and put
\[
    r_j=\dim E_j.
\]
Thus
\begin{equation}
\label{eq:Aj-upper}
    A|_{E_j}\preceq 2^j\Id_{E_j},
    \qquad j\geq0.
\end{equation}
For $j\geq1$, every principal axis belonging to $E_j$ has squared
length larger than $2^{j-1}$. Since
\[
    \sum_{i=1}^n\lambda_i(A)
    =
    \int_{\R^n}|x|^2\,d\mu_t(x)
    \leq n
\]
by \eqref{eq:trace-bound-integral}, we obtain that for all $j\geq 1$
\begin{equation}
\label{eq:rj-decay}
    r_j
    \leq
    2^{1-j}n.
\end{equation}
The marginal $(P_j)_*\mu_t$ is again $t$-uniformly log-concave.
Indeed, writing the density of $\mu_t$ as
\[
    e^{-t|x|^2/2}g(x),
\]
with $g$ log-concave, its marginal on $E_j$ has density proportional
to
\[
    e^{-t|z|^2/2}
    \int_{E_j^\perp}
    e^{-t|w|^2/2}g(z+w)\,dw,
\]
and the second factor is log-concave by Pr\'ekopa's theorem.
Together with \eqref{eq:Aj-upper}, Klartag's improved Lichnerowicz
inequality gives
\begin{equation}
\label{eq:CP-Ej}
    C_P\bigl((P_j)_*\mu_t\bigr)
    \leq
    \sqrt{\frac{2^j}{t}}.
\end{equation}
Since $B=(\Id+A)^{-1/2}$ is a function of $A$, each $E_j$ is also
$B$-invariant. Let $B_j$ denote the restriction of $B$ to $E_j$.
Then
\[
    \langle Bx,x\rangle
    =
    \sum_{j\geq0}
    \langle B_jP_jx,P_jx\rangle.
\]
Consequently, by Minkowski's inequality,
\begin{align}
\label{eq:U-Minkowski}
    U
    \leq
    \sum_{j\geq0}
    \Bigg[
    \int_{\R^n}
    \bigg(
        \langle B_jP_jx,P_jx\rangle
        -
        \int_{\R^n}
        \langle B_jP_jy,P_jy\rangle\,d\mu_t(y)
    \bigg)^2
    d\mu_t(x)
    \Bigg]^{1/2}.
\end{align}
Applying the Poincar\'e inequality
\eqref{eq:CP-Ej} to the function
\[
    z\longmapsto\langle B_jz,z\rangle
\]
gives
\begin{align}
\label{eq:quadratic-Ej}
    \int_{\R^n}
    \bigg(
        \langle B_jP_jx,P_jx\rangle
        -
        \int_{\R^n}
        \langle B_jP_jy,P_jy\rangle\,d\mu_t(y)
    \bigg)^2
    d\mu_t(x)
    \leq
    4\sqrt{\frac{2^j}{t}}
    \int_{\R^n}|B_jP_jx|^2\,d\mu_t(x).
\end{align}
If $\lambda$ is an eigenvalue of $A$ on $E_j$, the corresponding
eigenvalue of $B$ is $(1+\lambda)^{-1/2}$. 
Let $(e_i)_{i=1}^n$ be an orthonormal eigenbasis of $A_t$, so that
\[
A_te_i=\lambda_i e_i.
\]
Since
\(
B=(\Id+A_t)^{-1/2},
\)
we also have
\[
Be_i=(1+\lambda_i)^{-1/2}e_i.
\]
Therefore
\[
|BP_jx|^2
=
\sum_{i:\,e_i\in E_j}
\frac{\langle x,e_i\rangle^2}{1+\lambda_i}.
\]
Using that $\mu_t$ is centered and $A_t=\Cov(\mu_t)$,
\[
\int_{\R^n}\langle x,e_i\rangle^2\,d\mu_t(x)
=
\lambda_i,
\]
and hence
\[
\int_{\R^n}|BP_jx|^2\,d\mu_t(x)
=
\sum_{i:\,e_i\in E_j}
\frac{\lambda_i}{1+\lambda_i}
\leq r_j.
\]

Hence
\[
    \int_{\R^n}|B_jP_jx|^2\,d\mu_t(x)
    =
    \sum_{\lambda_i\in E_j}
    \frac{\lambda_i}{1+\lambda_i}
    \leq
    r_j.
\]
It follows from \eqref{eq:quadratic-Ej} that
\begin{equation}
\label{eq:Uj-bound}
    \Bigg[
    \int_{\R^n}
    \bigg(
        \langle B_jP_jx,P_jx\rangle
        -
        \int_{\R^n}
        \langle B_jP_jy,P_jy\rangle\,d\mu_t(y)
    \bigg)^2
    d\mu_t(x)
    \Bigg]^{1/2}
    \leq
    2t^{-1/4}2^{j/4}\sqrt{r_j}.
\end{equation}

For $j=0$, this is bounded by
\[
    2t^{-1/4}\sqrt n.
\]
For $j\geq1$, using \eqref{eq:rj-decay},
\[
    2^{j/4}\sqrt{r_j}
    \leq
    \sqrt{2n}\,2^{-j/4}.
\]
Therefore \eqref{eq:U-Minkowski} yields
\begin{align}
    U
    &\leq
    Ct^{-1/4}\sqrt n
    \left(
        1+\sum_{j\geq1}2^{-j/4}
    \right)
    \nonumber\\
    &\leq
    Ct^{-1/4}\sqrt n.
\end{align}
Thus
\begin{equation}
\label{eq:U-final}
    U
    \leq
    Ct^{-1/4}\sqrt n.
\end{equation}
We estimate $V$ directly, without decomposing the covariance
ellipsoid. By \eqref{eq:BL-covariance},
\[
    \|A_t\|_{\mathrm{op}}\leq t^{-1}.
\]
Klartag's improved Lichnerowicz inequality applied to $\mu_t$ itself
therefore gives
\[
    C_P(\mu_t)
    \leq
    \sqrt{
        \frac{\|A_t\|_{\mathrm{op}}}{t}
    }
    \leq
    \frac1t.
\]
Applying Poincar\'e to the function $x\mapsto|x|^2$, and using
\eqref{eq:trace-bound-integral}, we obtain
\begin{align}
    V^2
    &\leq
    \frac1t
    \int_{\R^n}
    \left|\nabla |x|^2\right|^2\,d\mu_t(x)
    \nonumber\\
    &=
    \frac4t
    \int_{\R^n}|x|^2\,d\mu_t(x)
    \leq
    \frac{4n}{t}.
\end{align}
Hence
\begin{equation}
\label{eq:V-final}
    V
    \leq
    2t^{-1/2}\sqrt n.
\end{equation}

Combining \eqref{eq:Phi-CS}, \eqref{eq:U-final}, and
\eqref{eq:V-final}, we conclude that
\[
    |\Phi'(t)|
    \leq
    Cn\,t^{-3/4},
\]
which proves \eqref{eq:Phi-derivative-bound}.
\end{proof}

\begin{theorem}
\label{thm:central-survival}
There exist universal constants
\(
    t_0,c_1,c_2,\beta>0
\)
such that the following holds.

For every even isotropic log-concave probability measure $\mu$ on
$\R^n$, there exists a subspace $F\subseteq\R^n$ satisfying
\[
    \dim F\geq\beta n
\]
and
\begin{equation}
\label{eq:central-survival}
    c_1\Id_F
    \preceq
    A_{t_0}P_F
    \preceq
    c_2\Id_F.
\end{equation}
\end{theorem}

\begin{proof}
Since $A_0=\Id$,
\begin{equation}
\label{eq:Phi-zero}
    \Phi(0)
    =
    (\sqrt2-1)n.
\end{equation}
By Lemma~\ref{lem:central-differential},
\begin{align}
    \Phi(t)
    &\geq
    \Phi(0)
    -
    Cn\int_0^t s^{-3/4}\,ds
    \nonumber\\
    &=
    (\sqrt2-1)n
    -
    C'n\,t^{1/4}.
\end{align}
Choose a sufficiently small universal $t_0>0$ so that
\[
    C't_0^{1/4}
    \leq
    \frac{\sqrt2-1}{2}.
\]
Then
\begin{equation}
\label{eq:Phi-survives}
    \Phi(t_0)
    \geq
    c_0n,
    \qquad
    c_0=\frac{\sqrt2-1}{2}.
\end{equation}

Let $\lambda_1,\ldots,\lambda_n$ be the eigenvalues of $A_{t_0}$.
Choose a sufficiently small universal $\delta>0$ such that
\[
    \sqrt{1+\delta}-1
    \leq
    \frac{c_0}{2}.
\]
Set
\[
    I
    =
    \{i:\lambda_i\geq\delta\}.
\]
The eigenvalues outside $I$ contribute at most $c_0n/2$ to
$\Phi(t_0)$. Hence \eqref{eq:Phi-survives} gives
\begin{equation}
\label{eq:large-spectrum-contribution}
    \sum_{i\in I}
    \left(
        \sqrt{1+\lambda_i}-1
    \right)
    \geq
    \frac{c_0n}{2}.
\end{equation}
On the other hand, \eqref{eq:BL-covariance} gives
\[
    \lambda_i\leq t_0^{-1},
\]
and therefore
\[
    \sqrt{1+\lambda_i}-1
    \leq
    \sqrt{1+t_0^{-1}}-1
    =:C_0.
\]
Together with \eqref{eq:large-spectrum-contribution}, this implies
\[
    |I|
    \geq
    \frac{c_0}{2C_0}n
    =:\beta n.
\]
Let $F$ be the span of the corresponding principal directions of
$Z_2(\mu_{t_0})$. Then
\[
    \dim F\geq\beta n,
\]
and
\[
    \delta\Id_F
    \preceq
    A_{t_0}P_F
    \preceq
    t_0^{-1}\Id_F.
\]
Thus \eqref{eq:central-survival} holds with
\[
    c_1=\delta,
    \qquad
    c_2=t_0^{-1}.
\]
\end{proof}
\section{A universal small-ball estimate}

We next extract from Theorem~\ref{thm:central-survival} the only
small-ball information needed in the remainder of the proof.

\begin{corollary}
\label{cor:universal-small-ball}
There exist universal constants $\varepsilon_0,c>0$ such that every
centered isotropic log-concave random vector $X$ in $\R^m$ satisfies
\begin{equation}
\label{eq:universal-small-ball}
    \mathbb P\left\{
        |X|\leq\varepsilon_0\sqrt m
    \right\}
    \leq e^{-cm}.
\end{equation}
\end{corollary}

\begin{proof}
We first assume that the law $\rho$ of $X$ is even.

Apply Theorem~\ref{thm:central-survival} to $\rho$ in dimension $m$.
There exists a subspace $F\subseteq\R^m$, with
\[
    k:=\dim F\geq\beta m,
\] such that for the marginal $\nu$ of $\rho_{t_0}$ onto $F$ we have
\begin{equation}
\label{eq:eta-covariance}
    c_1\Id_F
    \preceq
    \Cov(\nu)
    \preceq
    c_2\Id_F.
\end{equation}
Since $\rho$ is even, so are $\rho_{t_0}$ and $\nu$. In particular,
$\nu$ is centered. Let $Y$ have distribution $\nu$. Then
\[
    \E|Y|^2
    =
    \Tr\Cov(\nu)
    \geq
    c_1 k.
\]
By reverse H\"older's inequality for seminorms of
log-concave random vectors,
\[
    \bigl(\E|Y|^2\bigr)^{1/2}
    \leq
    C\,\E|Y|.
\]
Consequently, there exists a universal constant $a>0$ such that
\begin{equation}
\label{eq:EY-lower}
    \E|Y|\geq a\sqrt{k}.
\end{equation}

Moreover, $\nu$ is $t_0$-uniformly log-concave. Indeed,
$\rho_{t_0}$ is $t_0$-uniformly log-concave, and this property is
preserved under taking marginals by Pr\'ekopa's theorem. Hence the
Gaussian concentration inequality for uniformly log-concave
measures gives, for every $r>0$,
\[
    \nu\left\{
        |Y|\leq \E|Y|-r
    \right\}
    \leq
    \exp(-c t_0r^2).
\]
Taking
\[
    r=\frac{a}{2}\sqrt{k}
\]
and using \eqref{eq:EY-lower}, we obtain
\begin{equation}
\label{eq:nu-small-ball}
    \nu\left(
        \frac{a}{2}\sqrt{k}\,B_2^F
    \right)
    \leq
    e^{-c_0k},
\end{equation}
where $c_0>0$ is universal.

Choose now a universal $\varepsilon_0>0$ sufficiently small so that
\begin{equation}
\label{eq:epsilon-choices}
    \varepsilon_0
    \leq
    \frac{a\sqrt{\beta}}{2}
    \qquad\text{and}\qquad
    \frac{t_0\varepsilon_0^2}{2}
    \leq
    \frac{c_0\beta}{2}.
\end{equation}
Since $k\geq\beta m$, the first condition in
\eqref{eq:epsilon-choices} implies
\[
    \varepsilon_0\sqrt m
    \leq
    \frac{a}{2}\sqrt{k}.
\]
Therefore, using that orthogonal projection does not increase the
Euclidean norm,
\begin{align}
\label{eq:damped-small-ball}
    \rho_{t_0}
    \left(
        \varepsilon_0\sqrt m\,B_2^m
    \right)
    &\leq
    \nu
    \left(
        \varepsilon_0\sqrt m\,B_2^F
    \right)
    \nonumber\\
    &\leq
    \nu
    \left(
        \frac{a}{2}\sqrt{k}\,B_2^F
    \right)
    \nonumber\\
    &\leq
    e^{-c_0k}
    \leq
    e^{-c_0\beta m}.
\end{align}

We now return from the measure $\rho_{t_0}$ to $\rho$.
By definition,
\[
    d\rho_{t_0}(x)
    =
    \frac{e^{-t_0|x|^2/2}}{Z_{t_0}}\,d\rho(x),
\]
and hence
\[
    d\rho(x)
    =
    Z_{t_0}e^{t_0|x|^2/2}\,d\rho_{t_0}(x).
\]
Since $Z_{t_0}\leq1$, on the ball
$\varepsilon_0\sqrt m\,B_2^m$ we have
\[
    d\rho
    \leq
    e^{t_0\varepsilon_0^2m/2}\,d\rho_{t_0}.
\]
Together with \eqref{eq:damped-small-ball} and the second condition
in \eqref{eq:epsilon-choices}, this gives
\begin{align}
    \rho\left(
        \varepsilon_0\sqrt m\,B_2^m
    \right)
    &\leq
    \exp\left(
        \frac{t_0\varepsilon_0^2}{2}m-c_0\beta m
    \right)
    \nonumber\\
    &\leq
    e^{-c_0\beta m/2}.
\end{align}
Thus \eqref{eq:universal-small-ball} holds when $\rho$ is even.

We finally remove the symmetry assumption. Let $X'$ be an
independent copy of $X$ and set
\[
    W=\frac{X-X'}{\sqrt2}.
\]
The random vector $W$ is even and isotropic. Its law is also
log-concave, since log-concavity is preserved under reflection,
convolution, and linear images. Hence the even case gives
\[
    \mathbb P\left\{
        |W|\leq\varepsilon_0\sqrt m
    \right\}
    \leq
    e^{-cm}.
\]
On the other hand,
\[
    \left\{
        |X|\leq\frac{\varepsilon_0}{\sqrt2}\sqrt m,\,
        |X'|\leq\frac{\varepsilon_0}{\sqrt2}\sqrt m
    \right\}
    \subseteq
    \left\{
        |W|\leq\varepsilon_0\sqrt m
    \right\}.
\]
By independence,
\[
    \mathbb P\left\{
        |X|\leq\frac{\varepsilon_0}{\sqrt2}\sqrt m
    \right\}^2
    \leq
    e^{-cm}.
\]
Taking square roots and renaming the universal constants yields
\eqref{eq:universal-small-ball}.
\end{proof}

For completeness, let us point out that the proof could already be
concluded at this stage by appealing to the reduction of Dafnis and
Paouris, which shows that an exponential small-ball estimate of the
form obtained above implies a universal bound for the isotropic
constant. We prefer not to use this reduction here, since its proof
relies on the theory of $M$-ellipsoids and Pisier's $\alpha$-regular
positions. Instead, in order to keep the remainder of the argument
essentially self-contained and geometric, we derive the slicing bound
directly from the small-ball estimate by means of a weighted
combination of $L_p$-centroid bodies, Legendre duality, and the reverse
Santal\'o inequality.

We now complete the proof of the slicing theorem. By affine
invariance, it is enough to consider a convex body
$K\subseteq\R^n$ whose uniform probability measure $\mu_K$ is
centered and satisfies
\begin{equation}
\label{eq:isotropic-final}
    \Cov(\mu_K)=\Id.
\end{equation}
Let $X\sim\mu_K$.

For $p\geq2$, recall that the $L_p$-centroid body $Z_p(K)$ is the
origin-symmetric convex body whose support function is
\[
    h_{Z_p(K)}(u)
    =
    \left(
        \int_K |\langle x,u\rangle|^p\,d\mu_K(x)
    \right)^{1/p}.
\]
Fix a sufficiently small universal constant $a>0$, and let $p_*$ be
the largest dyadic integer satisfying
\[
    p_*\leq an.
\]
We may clearly assume that $n$ is larger than a universal constant,
so that $p_*\geq4$.

Define the origin-symmetric convex body $R$ by
\begin{equation}
\label{eq:R-final}
    h_R(u)^2
    =
    |u|^2
    +
    \sum_{\substack{p=2^j\\4\leq p\leq p_*}}
    p\,h_{Z_p(K)}(u)^2.
\end{equation}
Set
\[
    K_s=\mathrm{conv}(K\cup(-K)).
\]
Then, 
\[
    h_{Z_p(K)}(u)
    \leq
    h_{K_s}(u)
\]
for every $p\geq2$. Moreover, by isotropicity,
\[
    h_{Z_2(K)}(u)^2
    =
    \int_K|\langle x,u\rangle|^2\,d\mu_K(x)
    =
    |u|^2,
\]
and hence
\[
    |u|\leq h_{K_s}(u).
\]
Since the sum in \eqref{eq:R-final} is dyadic,
\[
    \sum_{\substack{p=2^j\\4\leq p\leq p_*}}p
    \leq 2p_*
    \leq 2an.
\]
It follows that
\[
    h_R(u)^2
    \leq
    Cn\,h_{K_s}(u)^2,
\]
and therefore
\begin{equation}
\label{eq:R-in-Ks}
    R\subseteq C\sqrt n\,K_s.
\end{equation}

Since the barycenter of $K$ is the origin, we have $0\in K$. Thus
\[
    K_s=\mathrm{conv}(K\cup(-K))\subseteq K-K.
\]
Hence, by the Rogers--Shephard inequality,
\[
    |K_s|
    \leq
    |K-K|
    \leq
    \binom{2n}{n}|K|
    \leq
    4^n|K|.
\]
Consequently, \eqref{eq:R-in-Ks} gives
\[
    |R|^{1/n}
    \leq
    C\sqrt n\,|K_s|^{1/n}
    \leq
    C\sqrt n\,|K|^{1/n}.
\]
It is therefore enough to prove
\begin{equation}
\label{eq:R-volume-target}
    |R|^{1/n}\geq c\sqrt n.
\end{equation}
Indeed, \eqref{eq:R-volume-target} would imply
$|K|^{1/n}\geq c$, and since
\[
    L_K=|K|^{-1/n}
\]
under the normalization \eqref{eq:isotropic-final}, this would yield
$L_K\leq C$.

\begin{proposition}
\label{prop:R-volume}
For the body $R$ defined in \eqref{eq:R-final}, one has
\[
    |R|^{1/n}\geq c\sqrt n,
\]
where $c>0$ is a universal constant.
\end{proposition}

\begin{proof}
For every dyadic $p$ with $4\leq p\leq p_*$, set
\[
    \Phi_p(u)
    =
    \frac12 h_{Z_p(K)}(u)^2
    =
    \frac12
    \left(
        \int_K|\langle x,u\rangle|^p\,d\mu_K(x)
    \right)^{2/p}.
\]
Also define
\[
    \Phi(u)
    =
    \frac12h_R(u)^2
    =
    \frac12|u|^2
    +
    \sum_{\substack{p=2^j\\4\leq p\leq p_*}}
    p\,\Phi_p(u).
\]
Fix
$p\geq4$ and $u\neq0$, and write
\[
    m
    =
    \left(
        \int_K|\langle x,u\rangle|^p\,d\mu_K(x)
    \right)^{1/p}.
\]
For $v\in\R^n$, define
\[
    Q_{p,u}(v)
    =
    m^{2-p}
    \int_K
    |\langle x,u\rangle|^{p-2}
    \langle x,v\rangle^2\,d\mu_K(x)
\]
and
\[
    \ell_{p,u}(v)
    =
    m^{1-p}
    \int_K
    |\langle x,u\rangle|^{p-2}
    \langle x,u\rangle
    \langle x,v\rangle\,d\mu_K(x).
\]
A direct differentiation gives
\begin{equation}
\label{eq:Phi-p-Hessian-final}
    D^2\Phi_p(u)[v,v]
    =
    (p-1)Q_{p,u}(v)
    -
    (p-2)\ell_{p,u}(v)^2.
\end{equation}
By Cauchy--Schwarz,
\[
    \ell_{p,u}(v)^2
    \leq
    Q_{p,u}(v).
\]
In particular,
\begin{equation}
\label{eq:Phi-p-positive}
    D^2\Phi_p(u)\succeq0.
\end{equation}
We shall use the following lemma, which was also proved in
\cite{BrazitikosPandisMeanGauge}. For completeness, we give a
self-contained proof.
\begin{lemma}
There exists a subspace
$E_{p,u}\subseteq\R^n$ such that
\begin{equation}
\label{eq:Epu-codim-final}
    \codim E_{p,u}\leq C_0p
\end{equation}
and
\begin{equation}
\label{eq:Epu-curvature-final}
    D^2\Phi_p(u)[v,v]
    \geq
    c_0p|v|^2,
    \qquad
    v\in E_{p,u},
\end{equation}
where $c_0,C_0>0$ are universal.

\end{lemma}
\begin{proof} Set
\[
    W
    =
    \int_K
    |\langle x,u\rangle|^{p-2}\,d\mu_K(x).
\]
The one-dimensional reverse moment inequality for centered
log-concave random variables gives
\[
    \| \langle X,u\rangle\|_p
    \leq
    \frac{p}{p-2}
    \|\langle X,u\rangle\|_{p-2}.
\]
Therefore
\[
\begin{aligned}
    W
    =
    \|\langle X,u\rangle\|_{p-2}^{p-2}  
    &\geq
    \left(\frac{p-2}{p}\right)^{p-2}
    m^{p-2}
    \geq
    e^{-2}m^{p-2}.
\end{aligned}
\]
Define a probability measure $\nu$ on $K$ by
\[
    d\nu(x)
    =
    \frac{|\langle x,u\rangle|^{p-2}}{W}\,d\mu_K(x).
\]
Then, by the moment comparison
\begin{equation}
\label{eq:Q-weighted-final}
    Q_{p,u}(v)
    =
    \frac{W}{m^{p-2}}
    \int_K\langle x,v\rangle^2\,d\nu(x)
    \geq
    e^{-2}
    \int_K\langle x,v\rangle^2\,d\nu(x).
\end{equation}
Let $F$ be the spectral subspace of the quadratic form $Q_{p,u}$
corresponding to eigenvalues smaller than a sufficiently small
universal constant $b>0$, and put
\[
    r=\dim F.
\]
We shall show that
\[
    r\leq Cp.
\]
This is immediate if $r=0$, so assume for the moment that $r\geq1$.
Let $e_1,\ldots,e_r$ be an orthonormal basis of $F$.
From
\eqref{eq:Q-weighted-final},
\begin{align}
    \int_K|P_Fx|^2\,d\nu(x)
    &=
    \sum_{j=1}^r
    \int_K\langle x,e_j\rangle^2\,d\nu(x)
    \nonumber\\
    &\leq
    e^2\sum_{j=1}^r Q_{p,u}(e_j)
    \leq
    e^2br.
\label{eq:weighted-second-moment-final}
\end{align}
Let $\varepsilon_0>0$ be the universal constant from
Corollary~\ref{cor:universal-small-ball}, and fix
\[
    0<\delta<\varepsilon_0.
\]
Choosing $b>0$ sufficiently small, depending only on $\delta$,
Markov's inequality and
\eqref{eq:weighted-second-moment-final} give
\[
    \nu\left\{
        |P_Fx|>\delta\sqrt r
    \right\}
    \leq
    \frac{e^2b}{\delta^2}
    \leq
    \frac14.
\]
Hence if \(    A=
    \left\{
        x\in K:
        |P_Fx|\leq\delta\sqrt r
    \right\}\), we have 
\begin{equation}
\label{eq:nu-A-final}
    \nu(A)\geq\frac34.
\end{equation}
By the definition of $\nu$ and
the moment comparison,
\begin{align}
    \int_A
    |\langle x,u\rangle|^{p-2}\,d\mu_K(x)
    &=
    W\nu(A)
    \nonumber\\
    &\geq
    \frac34W
    \geq
    \frac{3}{4e^2}\,m^{p-2}.
\label{eq:weighted-A-lower-final}
\end{align}
On the other hand, H\"older's inequality gives
\begin{align}
    \int_A
    |\langle x,u\rangle|^{p-2}\,d\mu_K(x)
    &\leq
    \left(
        \int_A|\langle x,u\rangle|^p\,d\mu_K(x)
    \right)^{(p-2)/p}
    \mu_K(A)^{2/p}
    \nonumber\\
    &\leq
    m^{p-2}\mu_K(A)^{2/p}.
\label{eq:weighted-A-upper-final}
\end{align}
Comparing
\eqref{eq:weighted-A-lower-final} and
\eqref{eq:weighted-A-upper-final}, we obtain
\begin{equation}
\label{eq:A-lower-final}
    \mu_K(A)\geq e^{-Cp}.
\end{equation}

Since $\mu_K$ is centered and isotropic,
the random vector $P_FX$ is centered and isotropic in the
$r$-dimensional space $F$. It is also log-concave, by Pr\'ekopa's
theorem. Since $\delta<\varepsilon_0$,
Corollary~\ref{cor:universal-small-ball} therefore gives
\begin{equation}
\label{eq:A-upper-final}
    \mu_K(A)
    =
    \mathbb P\left\{
        |P_FX|\leq\delta\sqrt r
    \right\}
    \leq
    e^{-cr}.
\end{equation}
Together with \eqref{eq:A-lower-final}, this yields
\begin{equation}
\label{eq:r-less-p-final}
    r\leq Cp.
\end{equation}
We have thus proved in all cases that
\[
    r=\dim F\leq Cp.
\]
Now set
\[
    E_{p,u}=F^\perp\cap\ker\ell_{p,u}.
\]
Notice that $\ell_{p,u}$ is nonzero, since
\[
    \ell_{p,u}(u)
    =
    m^{1-p}
    \int |\langle x,u\rangle|^p\,d\mu_K(x)
    =
    m>0.
\]
Hence $\ker\ell_{p,u}$ has codimension one, and therefore
\[
    \codim E_{p,u}
    \leq r+1
    \leq Cp+1
    \leq C_0p.
\]
By the definition of $F$, for all  $v\in F^\perp$ we have
\[
    Q_{p,u}(v)\geq b|v|^2.
\]
Thus, for $v\in E_{p,u}$, we have $\ell_{p,u}(v)=0$, and consequently
\[
\begin{aligned}
    D^2\Phi_p(u)[v,v]
    =
    (p-1)Q_{p,u}(v)
\geq
    b(p-1)|v|^2
    \geq
    c_0p|v|^2.
\end{aligned}
\]
This proves \eqref{eq:Epu-codim-final} and
\eqref{eq:Epu-curvature-final}.
\end{proof}
We now combine the information coming from all dyadic scales.
Fix $u\neq0$ and put
\[
    H(u)=D^2\Phi(u).
\]
By \eqref{eq:Phi-p-positive},
\[
    H(u)
    =
    \Id
    +
    \sum_{\substack{p=2^j\\4\leq p\leq p_*}}
    p\,D^2\Phi_p(u)
    \succeq
    \Id.
\]
Moreover, for $v\in E_{p,u}$,
\begin{align}
    \langle H(u)v,v\rangle
    &\geq
    |v|^2
    +
    p\,D^2\Phi_p(u)[v,v]
    \nonumber\\
    &\geq
    \bigl(1+c_0p^2\bigr)|v|^2.
\label{eq:H-on-Epu-final}
\end{align}

Let
\[
    \lambda_1(H(u))
    \leq
    \cdots
    \leq
    \lambda_n(H(u))
\]
be the eigenvalues of $H(u)$. We use the following elementary
consequence of the min--max principle: if a symmetric matrix $H$
satisfies
\[
    \langle Hv,v\rangle\geq L|v|^2
\]
on a subspace of codimension at most $d$, then by the Courant-Fisher formula 
\[
    \#\{i:\lambda_i(H)<L\}\leq d.
\]
Applying this observation to \eqref{eq:H-on-Epu-final} and
\eqref{eq:Epu-codim-final}, we obtain
\begin{equation}
\label{eq:H-counting-final}
    \#\left\{
        i:
        \lambda_i(H(u))<1+c_0p^2
    \right\}
    \leq
    C_0p.
\end{equation}
We now fix the universal constant $a>0$ in the definition of $p_*$.
By decreasing it if necessary, we may assume that
\begin{equation}
\label{eq:a-choice-final}
    C_0a\leq\frac14
    \qquad\text{and}\qquad
    64c_0C_0^2a^2\leq1.
\end{equation}
Since $p_*$ is the largest dyadic integer satisfying $p_*\leq an$,
we have, for $n$ larger than a universal constant,
\begin{equation}
\label{eq:pstar-final}
    \frac{a}{2}n<p_*\leq an.
\end{equation}
In particular,
\[
    C_0p_*
    \leq
    C_0an
    \leq
    \frac n4.
\]
We claim that, for every $1\leq i\leq n$,
\begin{equation}
\label{eq:eigenvalue-quadratic-final}
    \lambda_i(H(u))
    \geq
    \frac{c_0a^2}{4}\,i^2.
\end{equation}
We distinguish three ranges.

First, suppose that
\[
    16C_0\leq i\leq C_0p_*.
\]
Choose a dyadic integer $p$ such that
\[
    \frac{i}{4C_0}
    \leq
    p
    <
    \frac{i}{2C_0}.
\]
Such a $p$ exists by dyadicity. Moreover, $p\geq4$, while
\[
    p
    <
    \frac{i}{2C_0}
    \leq
    \frac{p_*}{2},
\]
so $4\leq p\leq p_*$. Also
\[
    C_0p<\frac i2<i.
\]
Hence \eqref{eq:H-counting-final} implies
\[
    \lambda_i(H(u))
    \geq
    1+c_0p^2
    \geq
    c_0p^2
    \geq
    \frac{c_0}{16C_0^2}\,i^2.
\]
Since $C_0a\leq1/4$, we have
\[
    \frac{1}{16C_0^2}
    \geq
    \frac{a^2}{4},
\]
and therefore
\begin{equation}
\label{eq:eigenvalue-middle-final}
    \lambda_i(H(u))
    \geq
    \frac{c_0a^2}{4}\,i^2.
\end{equation}

Next, suppose that
\[
    i>C_0p_*.
\]
Applying \eqref{eq:H-counting-final} with $p=p_*$ gives
\[
    \lambda_i(H(u))
    \geq
    1+c_0p_*^2
    \geq
    c_0p_*^2.
\]
By \eqref{eq:pstar-final},
\[
    p_*>\frac a2n,
\]
and hence
\[
    \lambda_i(H(u))
    \geq
    \frac{c_0a^2}{4}\,n^2
    \geq
    \frac{c_0a^2}{4}\,i^2,
\]
since $i\leq n$.

Finally, if
\[
    1\leq i<16C_0,
\]
then $H(u)\succeq\Id$, and consequently
\[
    \lambda_i(H(u))\geq1.
\]
On the other hand, by \eqref{eq:a-choice-final},
\[
    \frac{c_0a^2}{4}\,i^2
    \leq
    \frac{c_0a^2}{4}(16C_0)^2
    =
    64c_0C_0^2a^2
    \leq1.
\]
Thus \eqref{eq:eigenvalue-quadratic-final} holds in this range as
well.

We have therefore proved, uniformly for every $u\neq0$, that
\[
    \lambda_i(H(u))
    \geq
    \frac{c_0a^2}{4}\,i^2,
    \qquad
    1\leq i\leq n.
\]
Consequently,
\begin{align}
    \det H(u)
    &=
    \prod_{i=1}^n\lambda_i(H(u))
    \\
    &\geq
    \left(\frac{c_0a^2}{4}\right)^n
    \prod_{i=1}^ni^2
    \\
    &=
    \left(\frac{c_0a^2}{4}\right)^n(n!)^2.
\end{align}
Using $n!\geq(n/e)^n$, we conclude that for $u\neq 0$
\begin{equation}
\label{eq:det-H-final}
    \det H(u)
    \geq
    \left(
        \frac{c_0a^2}{4e^2}\,n^2
    \right)^n.
\end{equation}
We now pass to the Legendre transform
\[
    \Psi=\Phi^*.
\]
Since
\[
    \Phi(u)=\frac12h_R(u)^2
\]
and $h_R=\|\cdot\|_{R^\circ}$, the dual norm of $h_R$ is the
Minkowski functional $\|\cdot\|_R$. Therefore
\begin{equation}
\label{eq:Legendre-norm-final}
    \Psi(x)
    =
    \frac12\|x\|_R^2.
\end{equation}
The Euclidean term in the definition of $\Phi$ gives that for $u\neq 0$
\[
    D^2\Phi(u)\succeq\Id.
\]
Thus $\Phi$ is strictly convex, and $\nabla\Phi$ is invertible away
from the origin, with inverse $\nabla\Psi$. At corresponding points
\[
    x=\nabla\Phi(u),
    \qquad
    u=\nabla\Psi(x),
\]
differentiating these inverse relations yields
\[
    D^2\Psi(x)
    =
    \left[
        D^2\Phi(\nabla\Psi(x))
    \right]^{-1}.
\]
Hence \eqref{eq:det-H-final} gives that for  $x\neq0$
\begin{equation}
\label{eq:Psi-determinant-final}
    \det D^2\Psi(x)
    \leq
    \left(\frac{C}{n^2}\right)^n,
\end{equation}
We next identify the image of $R$ under $\nabla\Psi$. By
\eqref{eq:Legendre-norm-final},
\[
    R
    =
    \left\{
        x:\Psi(x)\leq\frac12
    \right\},
\]
whereas
\[
    R^\circ
    =
    \left\{
        u:\Phi(u)\leq\frac12
    \right\}.
\]
If
\[
    u=\nabla\Psi(x),
\]
then Euler's identity for the $2$-homogeneous function $\Psi$ gives
\[
    \langle x,u\rangle
    =
    \langle x,\nabla\Psi(x)\rangle
    =
    2\Psi(x).
\]
On the other hand, equality in Legendre duality gives
\[
    \Phi(u)+\Psi(x)
    =
    \langle x,u\rangle.
\]
Consequently,
\begin{equation}
\label{eq:equal-levels-final}
    \Phi(u)=\Psi(x).
\end{equation}
Since $\nabla\Phi$ and $\nabla\Psi$ are inverse maps,
\eqref{eq:equal-levels-final} shows that
\begin{equation}
\label{eq:gradient-image-final}
    \nabla\Psi(R)=R^\circ.
\end{equation}
\begin{comment}
All the functions involved are smooth away from the origin: here the
parameters $p$ are even integers and the Euclidean term makes the
Hessian positive definite. Hence $\nabla\Psi$ is a $C^1$
diffeomorphism between the interiors of $R\setminus\{0\}$ and
$R^\circ\setminus\{0\}$. 
\end{comment}
The
change-of-variables formula, together with
\eqref{eq:gradient-image-final}, gives
\begin{align}
    |R^\circ|
    &=
    \int_R
    \det D^2\Psi(x)\,dx
    \nonumber\\
    &\leq
    \left(\frac{C}{n^2}\right)^n|R|.
\label{eq:Rpolar-volume-final}
\end{align}
Finally, the Bourgain--Milman reverse Santal\'o inequality yields
\[
    |R|\,|R^\circ|
    \geq
    c^n|B_2^n|^2.
\]
Combining this with \eqref{eq:Rpolar-volume-final}, we obtain
\[
    c^n|B_2^n|^2
    \leq
    \left(\frac{C}{n^2}\right)^n|R|^2.
\]
Therefore
\[
    |R|^{1/n}
    \geq
    cn\,|B_2^n|^{1/n}.
\]
Since
\[
    |B_2^n|^{1/n}\simeq\frac1{\sqrt n},
\]
we conclude that
\[
    |R|^{1/n}\geq c\sqrt n.
\]
This proves the proposition.
\end{proof}
\paragraph*{Acknowledgements.}
The author is grateful to Apostolos Giannopoulos for valuable comments and suggestions that helped clarify the presentation of the result and streamline the exposition.
\bibliographystyle{plainurl}
\bibliography{slicing_references}

@incollection{BakryEmery1985,
  author    = {Bakry, Dominique and {\'E}mery, Michel},
  title     = {Diffusions hypercontractives},
  booktitle = {S{\'e}minaire de Probabilit{\'e}s XIX, 1983/84},
  editor    = {Az{\'e}ma, Jacques and Yor, Marc},
  series    = {Lecture Notes in Mathematics},
  volume    = {1123},
  pages     = {177--206},
  publisher = {Springer},
  address   = {Berlin},
  year      = {1985},
  doi       = {10.1007/BFb0075847}
}

@article{Klartag2006ConvexPerturbations,
  author  = {Klartag, Bo'az},
  title   = {On convex perturbations with a bounded isotropic constant},
  journal = {Geometric and Functional Analysis},
  volume  = {16},
  number  = {6},
  pages   = {1274--1290},
  year    = {2006},
  doi     = {10.1007/s00039-006-0588-1}
}

@article{Murawski2026,
  author  = {Murawski, Daniel},
  title   = {Comparing moments of real log-concave random variables},
  journal = {Bernoulli},
  volume  = {32},
  number  = {3},
  year    = {2026},
  doi     = {10.3150/25-BEJ1958}
}

@misc{BrazitikosPandisMeanGauge,
  author = {Brazitikos, Silouanos and Pandis, Christos},
  title  = {Geometric Bounds for the Mean Gauge and the Mean Width in Isotropic Position},
  year   = {2026},
  note   = {Preprint},
  url    = {https://arxiv.org/abs/2608.07216}
}

@book{Ledoux2001,
  author    = {Ledoux, Michel},
  title     = {The Concentration of Measure Phenomenon},
  series    = {Mathematical Surveys and Monographs},
  volume    = {89},
  publisher = {American Mathematical Society},
  address   = {Providence, RI},
  year      = {2001}
}

@article{Lewis1996,
  author  = {Lewis, Adrian S.},
  title   = {Derivatives of spectral functions},
  journal = {Mathematics of Operations Research},
  volume  = {21},
  number  = {3},
  pages   = {576--588},
  year    = {1996},
  doi     = {10.1287/moor.21.3.576}
}

@book{Rockafellar1970,
  author    = {Rockafellar, R. Tyrrell},
  title     = {Convex Analysis},
  series    = {Princeton Mathematical Series},
  volume    = {28},
  publisher = {Princeton University Press},
  address   = {Princeton, NJ},
  year      = {1970}
}

@misc{Bizeul2025,
  author = {Bizeul, Pierre},
  title  = {The slicing conjecture via small ball estimates},
  year   = {2025},
  note   = {Preprint; to appear in \emph{Annals of Probability}},
  url    = {https://arxiv.org/abs/2501.06854}
}

@article{Bourgain1986,
  author  = {Bourgain, Jean},
  title   = {On high-dimensional maximal functions associated to convex bodies},
  journal = {American Journal of Mathematics},
  volume  = {108},
  number  = {6},
  pages   = {1467--1476},
  year    = {1986},
  doi     = {10.2307/2374532}
}

@incollection{Bourgain1991,
  author    = {Bourgain, Jean},
  title     = {On the distribution of polynomials on high-dimensional convex sets},
  booktitle = {Geometric Aspects of Functional Analysis (1989--90)},
  series    = {Lecture Notes in Mathematics},
  volume    = {1469},
  pages     = {127--137},
  publisher = {Springer},
  address   = {Berlin},
  year      = {1991},
  doi       = {10.1007/BFb0089219}
}

@article{BrascampLieb1976,
  author  = {Brascamp, Herm Jan and Lieb, Elliott H.},
  title   = {On extensions of the {Brunn--Minkowski} and {Pr{\'e}kopa--Leindler} theorems, including inequalities for log concave functions, and with an application to the diffusion equation},
  journal = {Journal of Functional Analysis},
  volume  = {22},
  number  = {4},
  pages   = {366--389},
  year    = {1976},
  doi     = {10.1016/0022-1236(76)90004-5}
}

@book{BGVV2014,
  author    = {Brazitikos, Silouanos and Giannopoulos, Apostolos and Valettas, Petros and Vritsiou, Beatrice-Helen},
  title     = {Geometry of Isotropic Convex Bodies},
  series    = {Mathematical Surveys and Monographs},
  volume    = {196},
  publisher = {American Mathematical Society},
  address   = {Providence, RI},
  year      = {2014},
  doi       = {10.1090/surv/196}
}

@misc{ChenKlartag2026,
  author = {Chen, Yuansi and Klartag, Bo'az},
  title  = {Digesting the proof of the sharp thin-shell inequality},
  year   = {2026},
  note   = {Preprint},
  url    = {https://arxiv.org/abs/2607.23307}
}

@article{DafnisPaouris2010,
  author  = {Dafnis, Nikos and Paouris, Grigoris},
  title   = {Small ball probability estimates, {$\psi_2$}-behavior and the hyperplane conjecture},
  journal = {Journal of Functional Analysis},
  volume  = {258},
  number  = {6},
  pages   = {1933--1964},
  year    = {2010},
  doi     = {10.1016/j.jfa.2009.06.038}
}

@article{GiannopoulosPafisTziotziou2025,
  author  = {Giannopoulos, Apostolos and Pafis, Minas and Tziotziou, Natalia},
  title   = {The isotropic constant in the theory of high-dimensional convex bodies},
  journal = {Bulletin of the Hellenic Mathematical Society},
  volume  = {69},
  pages   = {89--188},
  year    = {2025}
}

@misc{Guan2024,
  author = {Guan, Qingyang},
  title  = {A note on Bourgain's slicing problem},
  year   = {2024},
  note   = {Preprint},
  url    = {https://arxiv.org/abs/2412.09075}
}

@article{Klartag2023,
  author  = {Klartag, Bo'az},
  title   = {Logarithmic bounds for isoperimetry and slices of convex sets},
  journal = {Ars Inveniendi Analytica},
  year    = {2023},
  pages   = {1--17},
  note    = {Paper No. 4}
}

@article{KlartagLehec2025,
  author  = {Klartag, Bo'az and Lehec, Joseph},
  title   = {Affirmative resolution of Bourgain's slicing problem using Guan's bound},
  journal = {Geometric and Functional Analysis},
  volume  = {35},
  number  = {4},
  pages   = {1147--1168},
  year    = {2025},
  doi     = {10.1007/s00039-025-00718-w}
}

@article{KlartagLehecSurvey2025,
  author  = {Klartag, Bo'az and Lehec, Joseph},
  title   = {Isoperimetric inequalities in high-dimensional convex sets},
  journal = {Bulletin of the American Mathematical Society},
  volume  = {62},
  number  = {4},
  pages   = {575--642},
  year    = {2025},
  doi     = {10.1090/bull/1869}
}

@article{GiannopoulosPaourisVritsiou2014,
  author  = {Giannopoulos, Apostolos and Paouris, Grigoris and Vritsiou, Beatrice-Helen},
  title   = {The isotropic position and the reverse {Santal{\'o}} inequality},
  journal = {Israel Journal of Mathematics},
  volume  = {203},
  number  = {1},
  pages   = {1--22},
  year    = {2014},
  doi     = {10.1007/s11856-012-0173-2}
}

\end{document}